\documentclass[11pt]{article}
\usepackage[utf8]{inputenc}
\usepackage[T1]{fontenc}
\usepackage{lmodern}
\usepackage{amsmath, amssymb, amsthm}
\usepackage{graphicx}
\usepackage{xcolor}
\usepackage{hyperref}
\usepackage{geometry}
\usepackage{epstopdf}
\usepackage{amsrefs}
\usepackage{relsize}
\usepackage{enumitem}
\usepackage{float}
\usepackage{wrapfig}
\usepackage{picinpar}
\usepackage{tcolorbox}
\usepackage{epsfig}

\theoremstyle{plain}
\newtheorem{theorem}{Theorem}
\newtheorem{proposition}[theorem]{Proposition}
\newtheorem{lemma}[theorem]{Lemma}
\newtheorem{corollary}[theorem]{Corollary}
\theoremstyle{definition}
\newtheorem{definition}[theorem]{Definition}

\def\n{\noindent}

\title{Hopf sets and the Universal Chord Theorem}
\author{Ion Ciudin \and Eugen J. Ionașcu}
\date{}

\begin{document}

\maketitle

\begin{abstract}
For a real-valued continuous function $f$ defined on the compact interval $[0,1]$, with $f(0)=f(1)$, we look at the set
$$\{\ell |\ell >0, \ell\not \in \bigcup_{t\in Range(f)} \{|a-b||a,b\in f^{-1}(t)\}\},$$
which leads to the concept of a Hopf set.  First, we study some topological properties of such sets. Maximal and minimal Hopf sets are introduced, studied and constructed. Some problems are formulated and we answer a recent question posed in the American Mathematical Monthly (2024) related to these objects. 
\end{abstract}

\vspace{10px}

\noindent{\textbf{Keywords:} {chords, Hopf sets, accumulation points, isolated points, additive open sets, null and conull Hopf sets, horizontal chords }} 

\vspace{10px}

\section{Introduction.} %




In this paper we are mostly interested in the concept of a Hopf set, which is at the heart of the Theorem~\ref{hopf} (Section~\ref{sec:history}),  a generalization of the Universal Chord Theorem (Theorem~\ref{UCT}). Let us start with a few definitions related to Hopf sets which are objects that need more attention in their own right.  
\begin{definition}\label{hopfset} Let us call a set $U$ of positive real numbers which is open, additive and distinct from $(0,\infty)$ a {\it Hopf set}. 
\end{definition} 
Of course some trivial examples are $(a,\infty)$ or $(a,b)\cup (2a,\infty)$, for $a>0$ and $a<b\le 2a$. Some less trivial examples are then needed. For all positive real numbers $a$ and $b$, so that $b\in (a,\frac{3a}{2}]$, we invite the reader to check that  the following three intervals set   $$(a,b)\cup (2b-a,2b)\cup (2b,\infty),$$
is a Hopf set. An example with four intervals 
$$(4,5)\cup (6,7)\cup (8,12)\cup (12,\infty),$$ 
may require a little more checking. The next example is quite intriguing 
\begin{equation}\label{newexample}
\small (25,26) \cup (40,44)\cup (45,48)\cup (50,52)\cup (55,56)\cup (60,74)\cup (75,100)\cup (100,\infty),
\end{equation}
and suggests a possible connection between these sets and number theory. One would like to provide an explicit construction of these Hopf sets, which the reader will find out by the end of this paper, it is not a trivial enterprise. Let's denote the class of Hopf sets by $\cal H$. 
We can see from these examples that an unbounded interval is contained in each of these sets. 
In fact this is true in general.  Indeed, since $U$ (a Hopf set) is open it contains an interval $I=(a,b)$, and we may assume $b<\infty$. If $n>\frac{a}{b-a}$ then for $k\ge n$, $kI$ overlaps with $(k+1)I$ since $(k+1)a<kb$ is equivalent to  $k>\frac{a}{b-a}$. Since $kI$ must be also in $U$, the observation follows. As a result of this, we can take $\alpha$ the maximum of the compact set $[0,\infty)\setminus U$, and observe that $\alpha_U$ satisfies

\begin{equation}\label{eqalpha}
0<\alpha_U:=\sup([0,\infty)\setminus U) \le \inf \{ \alpha_{(a,b)} | \ (a,b)\subset U\}, 
\end{equation}
\n where $\alpha_{(a,b)}:= (\max(\big\lfloor \frac{a}{b-a} \big \rfloor,1)+1) a$ and  as usual $\lfloor x\rfloor$ is the integer part of $x$ (real number). 
Hence, if we divide the set $U$ by the positive real number $\alpha$, we may assume that such a set is of the form $U=V\cup (1,\infty)$ with $V\subset (0,1)$. In other words, every Hopf set can be {\it normalized}.  Let's denote the set $[0,\alpha_U]\setminus U$ by $S(U)$. 
\begin{definition}\label{hopfsetmaxmin} A Hopf set $U$ is said to be a {\it null Hopf set} if the set $V$ in the normalization of $U$ is the empty set. A Hopf set $U$ is said to be a {\it conull Hopf set} if the set $V$ in the normalization has measure $\frac{1}{2}$. 
\end{definition} 
We will show that the measure of the set $V$ in this definition cannot be more than $\frac{1}{2}$. 
With the exception of $(a,\infty)$, all the examples above are conull Hopf sets.  It turns out that the set in the abstract, 
\begin{equation}\label{intermsoff}
H(f):=\{\ell |\ell >0, \ell\not \in \bigcup_{t\in Range(f)} \{|a-b||a,b\in f^{-1}(t)\}\},
\end{equation}
for a real-valued continuous function $f$ defined on the compact interval $[0,\alpha ]$, with $f(0)=f(\alpha)$,
is a Hopf set, and $\alpha$ is the one given in (\ref{eqalpha}) for $H(f)$.  We denote the class of such functions ${\cal F}(\alpha)$. This is in fact the characterization Theorem~\ref{hopf} (Section~\ref{sec:history}), where we will show the history of these ideas and a proof of this fact. In Section~\ref{nullHopfsets}, we study the functions  $f\in {\cal F}(\alpha)$ for which $H(f)$ is a null Hopf set.

The set 
$$S(f):=\bigcup_{t\in Range(f)} \{|a-b||a,b\in f^{-1}(t)\}$$
is called the {\it horizontal chords set} of $f$. We have clearly $H(f)=[0,\infty)\setminus S(f)$.
For a given Hopf set $U$ there are infinitely many functions $f$ so that $H(f)=U$, and so  we will discuss that in Section~\ref{nullHopfsets}. 
In Section~\ref{sec:isolated} we look at some topological properties of $V(U)$. Some constructions of Hopf sets are described in Section~\ref{constructions}. These constructions help us answer a question of  Davis and Troubetzkoy (\cite{Diana2024}) and a result about extending certain Hopf sets to a conull one appears in the end. Finally in Section~\ref{sec:symmetry} we are proving that the conull Hopf sets are symmetric in a special sense.

\section{A little bit of history.}\label{sec:history}
The story of the Universal Chord Theorem (UCT) is a long one and full of twists and turns, and it has recently resurfaced in the literature (see \cite{Diana2024}). It goes back to 1806, when André-Marie Ampère \cite{Ampere} proved the following equivalent statement:

\par \vspace{0.1in}
{\it \color{blue}  ``The graph of every continuous function on the compact interval $[0,1]$, with $f(0)=f(1)=0$, has at least one horizontal chord of length $\frac{1}{n}$, i.e., given $n$ a natural number, there exist $a$ and $b$ in $[0,1]$ such that $|a-b|=\frac{1}{n}$ and  $f(a)=f(b)$.''}
\par \vspace{0.1in}

\begin{figure}[htbp]
\centering
\epsfig{file=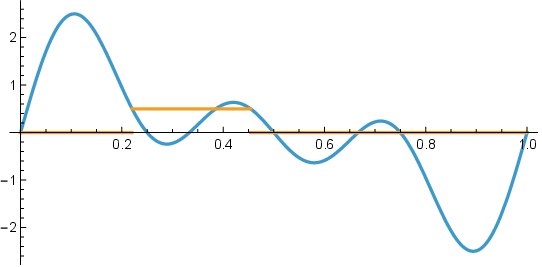,height=1in,width=2in}\qquad \epsfig{file=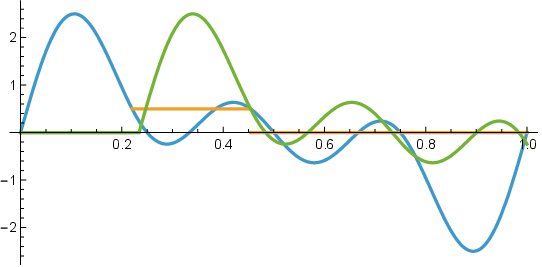,height=1in,width=2in}
\caption{Left: the graph of $f(x)=\sin(2\pi x)+\sin(4\pi x)+\sin(6\pi x)$ and a chord of length $a\approx 0.233843$. Right: the graphs of $y=f(x)$ and $y=f(x-a)$; the chords of length $a$ correspond to the points of intersection.}\label{fig:chord}
\end{figure}

In other words, every number of the form $\frac{1}{n}$, $n\in\mathbb N$ (our notation $\mathbb N$ stands for all positive integers), is the length of a horizontal chord on the graph of every such function. This is only half of what the standard UCT says; the other half is contained in Theorem~\ref{UCT} below, which is attributed to Lévy (\cite{Levy}), who proved it in 1934. In Figure~\ref{fig:chord} we have included an example and also a geometric way to think about the existence of a chord of a given length $a$: translate the graph of $y=f(x)$ to the right by $a$ and look for the intersections of the two graphs. We observe that in this case we have multiple points of intersection.

We will include a formulation of UCT which is close to the one in \cite{Burns} (Theorem 4)  and uses less notations. But we start with Heinz Hopf's result (\cite{Hopf})  from 1937 to get to the heart of the matter, although the UCT appeared a few years earlier. 

\begin{theorem}[Hopf]\label{hopf}
 The set $H(f)$ in (\ref{intermsoff})  is open, additive and it is not $(0,\infty)$. Given a set $U$ of positive numbers, which is open, additive, and not  $(0,\infty)$, there exists a function 
$f\in {\cal F}(L)$ ($L=\alpha_U$) such that  $H(f)=U$ 
\end{theorem} 

\begin{proof} For the first part, let us assume that $f$ is in ${\cal F}(1)$ without loss of generality. We take $a\in H(f)$ and we know that the equation $f(x-a)=f(x)$ has no solution for $x\in [a,1]$. 
 This means, the graphs of  $y=f(x)$ and $y=f(x-a)$ are compact sets in the plane which have no intersection. Then, modifying these two graphs to $y=f(x-a')$ and $y=f(x)$ on $[a',1]$, because $f$ is continuous, for $a'$ close enough to $a$, produces two graphs which are still a positive distance apart (do not intersect),  showing that $a'$ must be also in $H(f)$. As a result, $H(f)$ must be open. 
(We think this was Hopf's original idea.) For additivity, we will use the idea in (\cite{Diana2024}) which appears to be quite elegant.  

First, let us show that given a continuous periodic function $G$ on $\mathbb R$, and any real number $s$, the equation $G(x-s)=G(x)$ has at least one solution in $x$. By way contradiction if this is not true, we must have $G(x-s)>G(x)$ (or the other way around) for all $x$. This leads to a contradiction if we take $x$ to be a point of global maximum for $G$. If the inequality is $G(x-s)<G(x)$ for all $x$, we take $x$ a point of global minimum of $G$. 

By way of contradiction, let's take $a$ and $b$ in $H(f)$ and assume that $a+b$ is not in $H(f)$. Then $a+b\le 1$ and $a+b\in S(f)$. We can find $u,v\in [0,1]$ such that $f(u)=f(v)$ and $v-u=a+b$.
We can take the restriction of $f$ to the interval $[u,v]$ and extend it by periodicity to a function $G$. By the observation we made above the equation  $G(x-a)=G(x)$ has a solution, say $w$, which can be assumed to be in $[u,v]$ because of the periodicity of $G$ with period $v-u=a+b$. So, we have $G(w-a)=G(w)$ with $w\in [u,v]$. If $w-a$ is in $[u,v]$ the previous equality turns into $f(w-a)=f(w)$ which places $a$ in $S(f)$, a contradiction. It remains that $w-a<u$, and so $w+b-(w-a)=a+b=v-u$. This shows that $w+b\in [u,v]$ otherwise $w+b-(w-a)>v-u$ a contradiction. Hence, 
$$f(w)=G(w)=G(w-a)=G(w-a+a+b)=G(w+b)=f(w+b)$$ 
showing that $b\in S(f)$, again in contradiction with the assumption on $b$. It remains that $a+b\in H(f)$. 

For the second part, the existence, let us assume we have a Hopf set $U$ with $\alpha_U=1$. Then, $U=V\cup  (1,\infty)$ with $V$ open and $S(U)=[0,1]\setminus V$ a compact set. We let $K$ be the topological boundary of $S(U)$.  As a result, $S(U)=K\cup O$ with $O$ the interior points of $S(U)$. For $s\in [0,\infty)$ define the function $d(s)=\inf\{|s-u||u\in K\}$ (the distance to the compact set $S(U)$). Our function $f$ is then defined by
$$f(x)=\begin{cases} d(x)\ \ \text{ if} \ x\in U\\ \\ 
-d(x) \ \ \text{ if} \ x\in S(U).
\end{cases}$$ 
It is easy to show that $f$ is continuous at $x$, by looking at cases $x\in U$, $x\in K$ and $x\in O$. We observe that $f(t)=0$ for all $t\in K$. Also, if $t\in U$ we have $f(t)>0$ and for $t\in O$ we get $f(t)<0$. 

 ({\bf Case  I} )  First let's show that $S(U)=K\cup O \subset S(f)$. The simplest situation here is $K\subset S(f)$. Indeed, if $x\in K$ then $f(x)=f(0)=0$ and so $x$ is an horizontal chord length of $f$. Before we move on to the set $O$, let's prove a little lemma. 

\begin{lemma} \label{beginning} Given a Hopf set $U$ and an interval $[a,b]$ so that  $U\cap [a,b]=\emptyset$, then $[0,b-a]\cap U=\emptyset$. 
\end{lemma}
If $a=0$ there is nothing to prove. If there are no elements of $U$ less than $a$, there is nothing to prove either. Otherwise, let us the take largest interval $[A,b]$ such that $U\cap [A,b]=\emptyset$ ($A\le a$). Then $A$ is a boundary point of $U$. We denote by $z=x-A$ for an arbitrary point $x$ in $[A,b]$. We claim that $z\in S(U)$. If not, $z\in U$ and then $x=A+z$. Because $A$ is a boundary point of $U$
we can pick an $\epsilon>0$ such that  $A-\epsilon$ is in $U$ and of course, if $\epsilon$ is small enough $z+\epsilon\in U$ ($U$ being open). Since $U$ is additive we have the contradiction that $x=(A- \epsilon)+(z+\epsilon) \in V$. It remains that $z$ is in $S(U)$ and since $x$ was arbitrary in $A-b$ the whole interval 
$[0,b-A]$ is in $S(U)$. This proves the claim of the lemma.

Next, let us assume that $x\in O$ (an interior point of $S(U)$). By the definition of $f$, $f(x)<0$.

Consider the smallest point $y$ in $K$, such that $x<y$. Since the whole interval $[x,y]$ is in $S(U)$ using the Lemma~\ref{beginning}, we know that $w=y-x$ is in $S(U)$. We then consider $g(t)=f(t-x)-f(t)$ and compute $g(x)=f(0)-f(x)=-f(x)>0$, and $g(y)=f(w)-f(y)=f(w)\le 0$. By the Intermediate Value Theorem, we must have a point in $(x,y]$ say $t_0$ such that $g(t_0)=0$ or $f(t_0-x)=f(t_0)$ making $x$ a chord of $f$. In other words, we proved that $x\in S(f)$.

 ({\bf Case  II} ) Let us show that if $x\in U$, then 
\begin{equation}\label{remaining}
f(t-x)<f(t)\ \ \text{ for all}\  t\ge x.
\end{equation}
 This is enough to deduce  that $x\not\in S(f)$ and together with Case  I conclude that $S(f)=S(U)$. 
If $s=t-x$ is on the boundary $K$, then $t=s+x=(s\pm \epsilon)+(x\mp \epsilon)\in U$, if $\epsilon$ is small enough to insure $x\mp \epsilon)\in U$ and $s$ on the boundary insure that for infinitely many $\epsilon$ values satisfying the first restriction are in $U$. This places $t\in U$ and so the inequality (\ref{remaining}) becomes $0<f(t)$ which is true by the definition of $f$. 

If $s=t-x$ is in $U$, and as before $t=s+x\in U$. The inequality becomes $f(s)<f(s+x)$. Because $f(s)$ is the distance to the boundary $K$, say $f(s)=|s-k|$ with $k\in K$. This implies $(s-k,s+k)$ is in $U$.  Then adding $x\pm \theta\in U$ and using additivity gives $(s+x\pm \theta-k,s+x+\pm \theta+k)$ is in $U$ for every $\theta$ small enough. This shows that the point $s+x$ is at a greater distance to $K$, proving the desired inequality. 

Finally, if $s=t-x$ is in $O$. We split this is two cases. If $t\in K\cup U$, then $f(t)\ge 0>f(s)$ and we are done. If $t\in O$, the inequality can be written as $-f(s)>-f(t)$ or the distance to $K$ from $s$, $-f(s)$ is bigger than the distance to $K$ of $t=s+x$, $-f(t)$. Now the argument we used before does work with the points in $U$ which are close within an $\epsilon$ to $-f(s)$. So, given an arbitrary $\epsilon$ there exists $u\in U$ such that $|s-u|<-f(s)+\epsilon$ or $|s+x\pm \theta-(u+x\pm \theta)|<-f(s)+\epsilon$. This shows that the distance from $s+x\pm \theta$ is less than $-f(s)+\epsilon$ or the distance from $x+s$ is less than
 $-f(s)+\epsilon-\theta$. This gives the inequality we want if we fix $\theta$ so that $x\pm \theta\in U$, and $\epsilon<\theta$.  
\end{proof}

The proof of the existence here is inspired from the original paper of Hopf  (\cite{Hopf}) which is in German. We believe this is the first time when it appears in English. In  (\cite{Diana2024}) the authors only sketch the definition of $f$.

From now on, let us fix a positive real number $L>0$. We define the set $${\cal P}_L :=\{0\}\cup \{\frac{L}{n} : n\in \mathbb N\}.$$
 Let us observe that if we have a Hopf set $U$ ($\alpha_U=L$), then  we must have ${\cal P}_{L} \cap U=\emptyset$ because otherwise, if some $\frac{L}{n}\in U$,  by additivity $\frac{L}{n}+\frac{L}{n}+\cdots \frac{L}{n}=L \in U$, a contradiction.

For every $n\in \mathbb N$, we denote by $J_n(L)$ the open interval $(\frac{L}{n+1},\frac{L}{n})$. 

At this point, it appears that we need to make a small distinction of notations: 
$$kA:=A+A+\cdots +A=\{\sum_{i=1}^k a_i|\  a_i\in A\}$$
$$k\cdot A=\{ka|\ a\in A\}\ \ \ \text{for a set of real numbers}\ \ A. $$
In general $kA\not=k\cdot A$, but one can check that if $A$ is an interval these sets are the same.  Therefore, the set 

\begin{equation}\label{generichopf}
V_n(L)=\bigcup_{k=1}^n kJ_n(L) \subset (0,L), \ \text{which gives} \ U_n(L):=V_n(L)\cup (L,\infty).
\end{equation}
\n By definition $U_n(L)$ is additive, open and it is not $(0,\infty)$, so we have a Hopf set with $n+1$ open intervals. Its normalization is $U_n(1)$.   The measure of  $U_n(1)$ is 
$$\sum_{k=1}^n m(kJ_n(1))=\sum_{k=1}^n k\cdot \frac{1}{n(n+1)} =\frac{1}{2}, $$
showing that this is a conull Hopf set. 
This is saying, in particular, that 
\begin{equation}\label{eq1}
{\cal P}_L=\cap_{U\in {\cal H} } S(U)=\cap_{f\in {\cal F}(L) } S(f),
\end{equation}
and  UCT stated in the Theorem~\ref{UCT} below, follows from Theorem~\ref{hopf}.

\begin{theorem}[Universal Chord Theorem]\label{UCT}
For a given length $\ell$, the necessary and sufficient condition for a continuous function $f:[0,L]\to \mathbb R$ with $f(0)=f(L)$ to have a horizontal chord of length $\ell$, i.e. the preimage $f^{-1}(t)$ of some value $t$ contains two points at a distance of $\ell$ apart,  is that $\ell=\frac{L}{m}$ for some positive integer $m$. 
\end{theorem}

The necessity of the condition $\ell=\frac{L}{m}$ in Theorem~\ref{UCT} (i.e., the other inclusion in (\ref{eq1})) follows from Theorem~\ref{hopf}, but there is also a very simple explicit example due to Lévy \cite{Levy}. Take $L=1$ and $\ell\in(0,1)$ with $\ell\not=\frac{1}{n}$ for all $n\in\mathbb N$, and
$$f_{\ell}(x)=\left|\sin \left(\frac{\pi x}{\ell}\right) \right|-x \left|\sin \left(\frac{\pi }{\ell}\right) \right|.$$
This function is continuous and $f_{\ell}(0)=f_{\ell}(1)=0$. Because $|\sin t|$ is periodic with period $\pi$, we have for all $x\in [\ell , 1]$,
$$f_{\ell}(x-\ell)-f_{\ell}(x)=\ell \left|\sin \left(\frac{\pi }{\ell}\right) \right|\not =0,$$
since $\frac{1}{\ell}$ is not an integer. Hence $\ell\not\in S(f_{\ell})$ (see Figure~\ref{fig:levy}).

\begin{figure}[htbp]
\centering
\epsfig{file=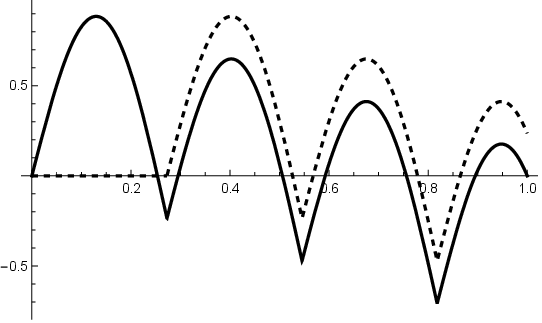,height=1in,width=2in}
\caption{L\'evy's function $f_{\ell}$ for $\ell=\frac{3}{11}$.}\label{fig:levy}
\end{figure}

In \cite{Burns}, one can find an elegant account of the UCT and a real world application of it in races. We include the following statement to give the reader a flavour of this kind of applications.

\begin{proposition}[\cite{Burns}] \label{Burns2017}
 If the race distance is a whole number of miles, then some mile must be covered at exactly the average pace. 
\end{proposition}
(Indeed, if $t(x)$ is the time at which the runner reaches the distance $x$, $x\in[0,n]$, and $T=t(n)$ is the total time, then $f(x)=t(x)-\frac{T}{n}x$ is continuous with $f(0)=f(n)=0$, and the UCT with $L=n$ and $\ell=1=\frac{L}{n}$ gives a mile which is covered in time exactly $\frac{T}{n}$.)

\par\vspace{0.1in}

Looking at the identity (\ref{eq1}), one may think that for every $n\in \mathbb N$, $n\ge 3$,  there exists a function $f_n\in {\cal F}(L)$ which has a chord set $S(f_n)$ with 

$$S(f_n) \subset [0,\frac{L}{n}] \cup \bigcup_{k=1}^{n-1} \{\frac{L}{k}\}.$$
But this cannot happen for the several reasons. We will show that the points $\frac{L}{n}$ cannot be isolated in $S(f)$. Also, this contradicts Hopf result about $H(f_n)$ being additive since $\frac{2L}{5}$ and $\frac{3L}{5}$ must be in this set but $\frac{2L}{5}+\frac{3L}{5}=L$ is not. Yet, another reason is shown in \cite{Diana2024}. 
\begin{theorem}[\cite{Diana2024}]
For every $f\in {\cal F}(1)$,  the chord set $S(f)$ must have measure greater than or equal to $\frac{1}{2}$. Equivalently, for every Hopf set $U$, the measure of $[0,\alpha_U]\cap U$ is at most $\frac{\alpha_U}{2}$. 
\end{theorem}
\begin{proof}
Let us introduce the symmetry $\phi$ across the point $\frac{1}{2}$, i.e., $\phi(x)=1-x$ with $x\in [0,1]$.  One observation here is that $\phi(V)$ and $V$ cannot intersect, and these are both subsets of $[0,1]$. Since $\phi$ preserves the Lebesque measure we get that $m(V)+m(\phi(V))\le m([0,1])=1$ or $2m(V)\le 1$. We conclude as in \cite{Diana2024} that $m(V)\le \frac{1}{2}$, where the authors use this idea of reflection across any point $d\in S(f)$. \end{proof}

\section{Some results related to null Hopf sets.}\label{nullHopfsets} 

In what follows we will assume that ${\cal F}_L$ is  the collection of continuous functions on $[0,L]$ with $f(0)=f(L)=0$. Given a Hopf set $U$ what can be said about the set

$${\cal U}:=\{f\in {\cal F}_L| \ H(f)=U \},$$
\n where $L=\alpha_U$?  First we observe that if $f\in {\cal U}$ then the functions $x\to -f(x)$, 
$x\to f(L-x)$, $x\to \lambda f(x)$ ($\lambda\in \mathbb R$) are also in ${\cal U}$. A more general situation is the following:  if we have a function $\psi: \mathbb R \to \mathbb R$ which is continuous and a bijection, then $\psi\circ f$ is also in ${\cal U}$. Indeed, if $(\psi\circ f)(a)=(\psi\circ f)(b)$ then this is equivalent to $ f(a)= f(b)$, and so the sets of chords are the same, i.e.,
$S(f)=S(\psi\circ f)$. 

We want to concentrate on the case when $U$ is a null Hopf set. In other words the set above is ${\cal U}_0:=\{f\in {\cal F}_L| \ S(f)=[0,L] \}$. For a set of real numbers $A$, we denote by $\Delta(A)$ the set of all numbers $|z-w|$ with $z,w\in A$. We need a small ingredient to study the scenario of having a discrete (and small) number of zeroes for $f$.
\begin{proposition}\label{prop1} If $a\in S(f)\setminus \Delta (f^{-1}(0))$ we must have $f([a,L])\cap f([0,L-a])\not =\{0\}$. 
\end{proposition}
\begin{proof} We know that $y=f(u)=f(v)$ for some $u$ and $v$ so that $|u-v|=u-v=a$, and we may assume that $y\not =0$.  
 Clearly $u\in [a,L]$ and $v=u-a\le L-a$, which shows the conclusion of the statement, since  
$y\in f([a,L])\cap f([0,L-a])$. \end{proof}
The converse of this proposition is not true. Indeed, we can take $$U_3(1)=(\frac{1}{3},\frac{1}{2})\cup (\frac{2}{3},\frac{2}{2})\cup (1,\infty)$$ and $f$ corresponding to this given by Theorem~\ref{hopf}. Taking $a=\frac{5}{12}\in U_3(1)$, one can check that $f([a,L])\cap f([0,L-a])=[-\frac{1}{12},\frac{1}{6}]\cup [-\frac{1}{6},\frac{1}{12}]=[-\frac{1}{12},\frac{1}{12}] \not = \{0\}$. 
An interesting corollary of this proposition is the following theorem. 

\begin{theorem}\label{prop2} If $S(f)=[0,L]$ and $f$ has $n$ changes of sign, i.e., $f(x_i)=0$, $i=0,1,2,\ldots,n,n+1$, where
$$x_0=0<x_1<x_2<\cdots<x_{n+1}=L$$ having alternating signs on each interval $(x_i,x_{i+1})$, $i=0,1,2,\ldots,n$, then $n$ must be even.
\end{theorem}
\begin{proof} By way of contradiction, if $n$ is odd, then it has opposite  signs on $(0,x_1)$ and on $(x_n,L)$. We take $a$ close enough to $L$ so that $[a,L)\subset (x_n,L)$ and $(0,L-a])\subset (0,x_1)$, and also smaller than all $|x_{i+1}-x_i|$, $i=0,1,2,\ldots,n$.  Hence, $f([a,L])\cap f([0,L-a]) =\{0\}$ and so $$a\not\in S(f)\setminus \Delta (f^{-1}(0))=[0,L]\setminus  \Delta (f^{-1}(0)),$$ by counterpositive of Proposition~\ref{prop1}. This means that $a\in  \Delta (f^{-1}(0))$ or $a=|u-v|$ with $f(u)=f(v)=0$, wich contradicts the last assumption on $a$. 
 \end{proof}
This theorem gives a necessary condition for a function $f$ having alternating signs like described above to be in ${\cal U}_0$ ($S(f)=[0,L]$). It is perhaps hopeless to find a characterization of these functions, but the following theorem is getting in that direction.

\begin{theorem}\label{prop3} If $f$ has $2$ changes of sign, i.e. $f(x_i)=0$, $i=0,1,2,3$, where
$$x_0=0<x_1<x_2<x_{3}=L$$ having alternating signs on each interval $(x_i,x_{i+1})$, $i=0,1,2,3$ beginning positive on $(0,x_1)$, and 
$$(L-x_1-x_2)(\underset{[x_2,L]}{max} f - \underset{[0,x_1]}{max} f )\le 0,$$
then $S(f)=[0,L]$.
\end{theorem}

\begin{figure}[htbp]
\centering
\epsfig{file=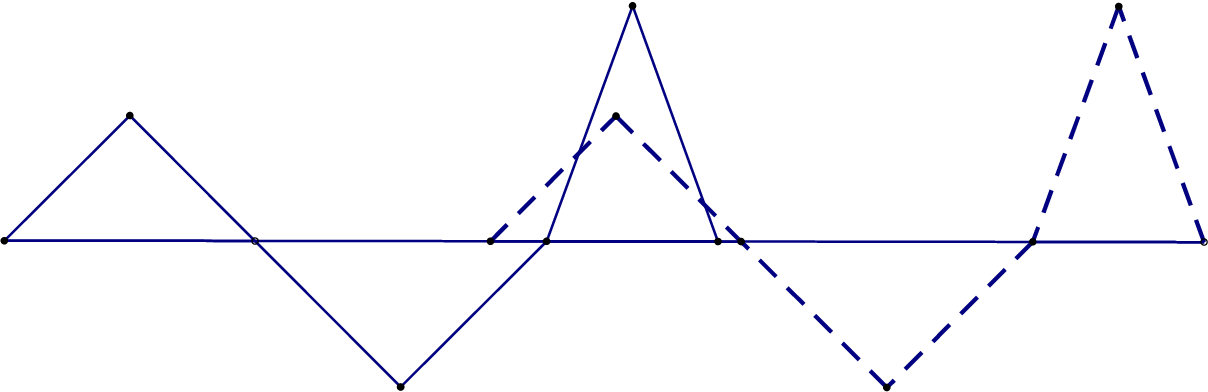,height=1in,width=2in}\qquad \epsfig{file=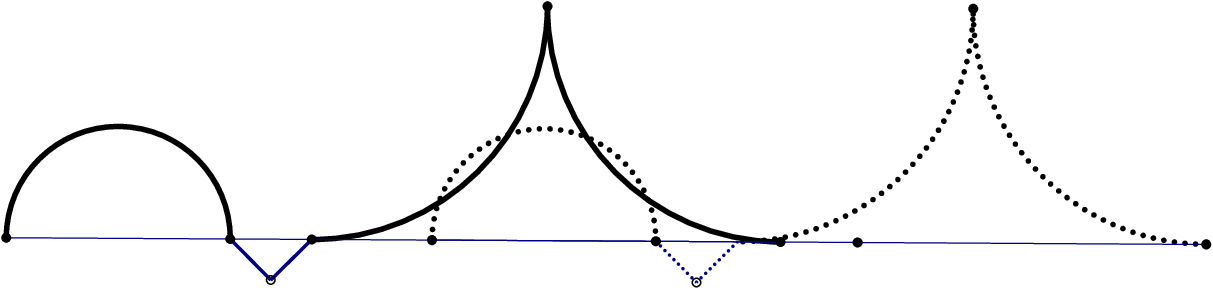,height=1in,width=2in}
\caption{Left: the situation in the proof of Theorem~\ref{prop3}. Right: a function for which the conclusion of Theorem~\ref{prop3} holds although the hypothesis fails.}\label{fig:prop3}
\end{figure}
\begin{proof} One can see that this is true from Figure~\ref{fig:prop3} (left). But let us show at least one of the cases involved here. We may assume first that $L-x_1-x_2<0$ or $L-x_2<x_1$ (see Figure~\ref{fig:prop3}, left), which by our assumption implies 
$$m_2=f(c_2)=\underset{[x_2,L]}{max} f \ge m_1=f(c_1)=\underset{[0,x_1]}{max} f  $$
for some $c_1\in (0,x_1)$ and $c_2\in (x_2,L)$. Assume $a\in (x_1,x_2)$ and let $g(x)=f(x-a)-f(x)$ for $x\in [a,L]$. We have $g(a)=f(0)-f(x)=0-f(x)>0$. On the other hand,
$$c_2\in [a,L]\implies 0\le c_2-a\le L-a \le L-x_1\le x_2 \implies f(c_2-a)\le m_1.$$
Hence, $g(c_2)=f(c_2-a)-f(c_2)\le m_1-m_2\le 0$. This shows that the equation $g(x)=0$ has a solution in $(a,c_2)$, and so $a\in S(f)$. The other cases can be treated is a similar fashion. 
Let us observe that if $x_1+x_2=L$ or $\underset{[x_2,L]}{max} f =\underset{[0,x_1]}{max} f$ the conclusion still stands. 
\end{proof}
The converse of the theorem is not true. One can see this by looking at the suggestion in Figure~\ref{fig:prop3} (right). 
If the set of zeros of $f$ is infinite there are wild things that can appear. Supposed $f$ is zero on $[0,\frac{1}{3}]\cup [\frac{2}{3},1]$ and arbitrary on $(\frac{1}{3},\frac{2}{3})$ but continuous. The it is easy to see that $f\in {\cal U}_0$.  In this instance the set of zeroes has positive measure. But even if we ask for zero measure a similar phenomenon happens.  An interesting class of  functions $f$ in ${\cal U}_0$ can be constructed starting with the Cantor set $C=[0,1]\setminus W$, where $W$ is a countable set of open intervals,  being $0$ on the Cantor set and extended by merely continuity on $W$. Then because $C-C=[-1,1]$ (see \cite{Edgar} and \cite{Falconer})   we must have $f\in {\cal U}_0$.

 From now on we will assume $L=1$ and drop all the indices that involve $L$ in general.

\section{Points in ${\cal P}\setminus \{1\}$ cannot be isolated in $S(f)$.}\label{sec:isolated}

In this paper we are also interested  in finding various properties of the chord set. One of them is about  the isolated points and the  accumulation points of $S(f)$.  We also investigate some possible ways to construct Hopf sets  and address a question posed in \cite{Diana2024}. 

In what follows we will assume that  $H(f)$  is $V\cup (1,\infty)$ for some open set $V$ in $(0,1)$ which can be empty. 

Lemma~ \ref{beginning}  shows that we have the following proposition.
\begin{proposition}\label{openinterval}
For $f\in \cal F$, there exists an interval $[0,\beta]$, with $\beta>0$, contained in $S(f)$.  
Equivalently, for every Hopf set $U$, there exists  $\beta>0$ such that $U\subset (\beta, \infty)$. 
\end{proposition}
\begin{proof} If $S(f)=[0,1]$ there is nothing to prove, so we may assume that $H(f)=V\cup (1,\infty)$ with $V$ nonempty. By way of contradiction, if $S(f)$ does not contain an interval of the form $[0,\beta]$, then for every $m\in \mathbb N$ there exists $x_m\in (0,\frac{1}{m})$ which is not in $S(f)$, i.e., $x_m\in H(f)$. We take a maximal interval $(a,b)$ contained in $V$. In other words $a, b\not\in V$, and so $b\not\in H(f)$. Since $x_m\to 0$, for $m$ big enough we have $x_m<b-a$ and so $b-x_m$ is in $(a,b)\subset H(f)$. Then $b=(b-x_m)+x_m\in H(f)$ by additivity, a contradiction. 
\end{proof}
In relation to this fact, there is a result of Levit (\cite{Levit}) from 1963 which states: if $f\in \cal F$ has $n$ changes of sign, then $\beta$ in this proposition can be taken to be $\frac{1}{\lfloor \frac{n+3}{2}\rfloor}$ (compare with Theorem~\ref{prop2}).

 So, this allows us to define for every $U\in {\cal H}$, $\beta(U):=\inf U>0$.

If we take $V=(1/2,1)$ we have $V\cup (1,\infty)$ a conull Hopf set. A function $f$ which has $S(f)=[0,1]\setminus V$ is $f(x)=\sin (2\pi x)$. So, we see that the point $1$ can be isolated in $S(f)$. 
\begin{theorem}\label{nonisolated}
 For $n\in \mathbb N$, $n\ge 2$ and $f\in \cal F$ , the point $\frac{1}{n}$ cannot be isolated in $S(f)$.
\end{theorem}
\begin{proof} We let $c=\frac{1}{n}$. If $c$ is isolated then there is a set $J=(a,c)\cup (c,b)$ in $H(f)$.  Because we can write $n=p+q$ with $p$ and $q$ positive integers, 
 $$1=nc=q(c-p\epsilon)+p(c+q\epsilon)$$ and for $\epsilon$ small enough we have $c-p\epsilon$ and $c+q\epsilon$ in $J$.  Since $H(f)$ is additive, we have $1$ written as linear combination with positive integer coefficients of two elements in $H(f)$, and so it follows that $1\in H(f)$. This is in contradiction with $1\not \in H(f)$. 
\end{proof}
We observe that we proved a little more here: if $\ell \in S(f)$ then $\frac{\ell}{n}$ cannot be isolated ($n\ge 2$).
What is interesting is that every point in $(0,1)\setminus {\cal P}$ can be isolated. Before we state and prove this let us introduce an important example of a maximal Hopf set. 

\begin{theorem}\label{isolatedpoints}
If $a \in (0,1) \setminus {\cal P}$,
then there exists $f \in {\cal F}$ such that $a \in S(f)$ and $a$ is an isolated point of $S(f)$. Moreover, one can find $f \in {\cal F}$ so that, in addition, $H(f)$ is a conull Hopf set.
\end{theorem}

\begin{proof}
By our assumption, there exists $n \in \mathbb{N}$ such that $\frac{1}{n+1} < a < \frac{1}{n}$. 
We consider the Hopf set $U_n$ defined above, and define $W_n=U_n\setminus \{a\}$. First we see that  $W_n$ is still open. Taking out a point of  $U_n$
in general does not preserve the additivity,  but in this case, there is no problem since the equation $x+y=a$, with $x$, $y$ in $U_n$ is not possible (we are removing a point from the first set in the union defining $U_n$ in (\ref{generichopf})).   
By Hopf's theorem there exists a function $f$ which has $S(f)=[0,1]\setminus W_n$. It is clear that $a$ is isolated in $S(f)$. The second part of the theorem follows from the fact that $W_n$ is a conull Hopf set. 
\end{proof}
\begin{figure}[htbp]
\centering
\epsfig{file=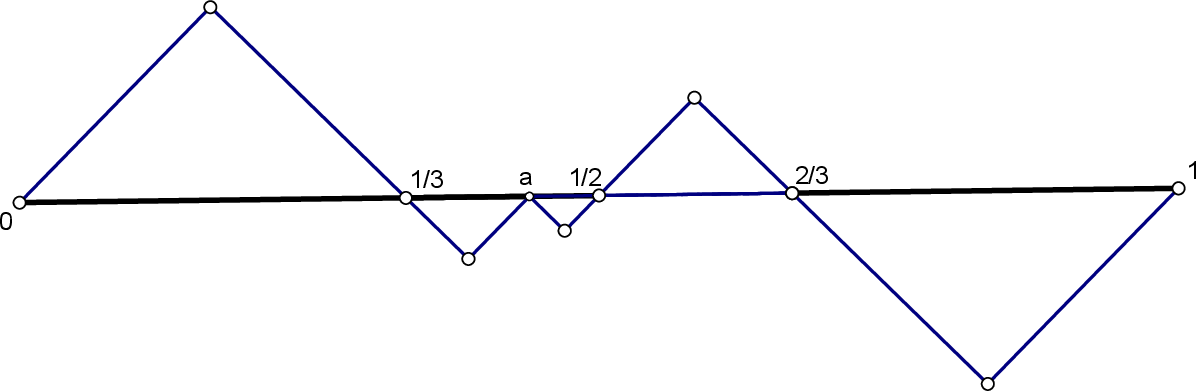,height=1in,width=2in}
\caption{The function given by Hopf's construction for $a=\frac{4}{9}\in (\frac{1}{3},\frac{1}{2})$.}\label{fig:th10}
\end{figure}

In Figure~\ref{fig:th10}, we are including such a standard construction of a function which follows Hopf's idea, with $H(f)=(\frac{1}{3},\frac{1}{2})\cup (  (\frac{2}{3},1)\setminus\{a\}$.

This topological analysis tells us more about open additive sets of non-negative real numbers, and in the next section we are going to analyze that in more detail.

\section{Possible ways to construct  $H(f)$.}\label{constructions}

The Hopf sets $U_n$ defined in the previous section play an important role in the construction of an open additive set $U$ of positive numbers. From Proposition~\ref{openinterval} we see that $U$ must be included in some interval $(\beta,\infty)$. We will assume that $\beta$ is the greatest number with this property. If $\beta=1$, then we have the $V=\emptyset$, and so $U=(1,\infty)$ which is a trivial example of such sets, in the same class with $U=(0,\infty)$. Let's call these cases exceptional.  

We justified that without loss of generality, if $U$ is not exceptional, we can assume that $U=V\cup (1,\infty)$, with $V$ nonempty $V\subsetneq (0,1)$. Because $1$ is not in the set $U$ then the whole set $\cal P$ cannot be in the set. So, we may assume that 

$$V\subset \left[(0,1)\setminus {\cal P}\right]   \cap (\beta,\infty).$$

Let us take the natural number $m$ so that $\frac{1}{m+1}\le\beta <\frac{1}{m}$. We denoted the interval $(\frac{1}{m+1},\frac{1}{m})$ by $J_m$. Let us define $W_k=U\cap k J_m$ with $k\in \{1,2,\cdots m\}$. We know that $W_1$ is non-empty and it is open. Then $U$ contains all the sets $$kW_1=W_1+W_1+\cdots W_1$$ with $k\in \{1,2,\cdots m\}$, but $W_k$ can be strictly bigger than $k\cdot W_1$.

\subsection{Case $\bf m=1$.}  This is a simple but important case in our constructions. We have $J_1=\left(\frac{1}{2}, 1\right)$ and because the sum of any two numbers in $J_1$ gives a number greater than $1$, the additivity is insured for $J_1$ but also for any subset of $J_1$.   As a result  we can choose a subset $Z$ to be the countable union of open intervals  that define a Cantor set (as in the standard construction on $[0,1]$), scaled to fit an interval $[a,b]$ so that $(a,b)\subset J_1$, with $m(Z)=b-a$. 
In particular, we can take $(a,b)=J_1$.\par\vspace{0.1in}

In \cite{Diana2024} there is this question about the existence of a function with countably many ``mountain ranges" and countably many ``valley ranges", giving points of relative maximum or minimum of $f$, and $m(S(f))=\frac{1}{2}$. For  a simple example, one can check that the following map does the trick:

\[
F(x) =
\begin{cases}
 -|x \sin\!\left(\frac{\pi}{x}\right)|, & 0 < x \le \tfrac{1}{2}, \\[6pt]
 |(1-x)\sin\!\left(\frac{\pi}{1-x}\right)|, & \tfrac{1}{2} <  x < 1, \\[6pt]
 0, & x \in \{0,1\}.
\end{cases}
\]
\begin{figure}[htbp]
\centering
\epsfig{file=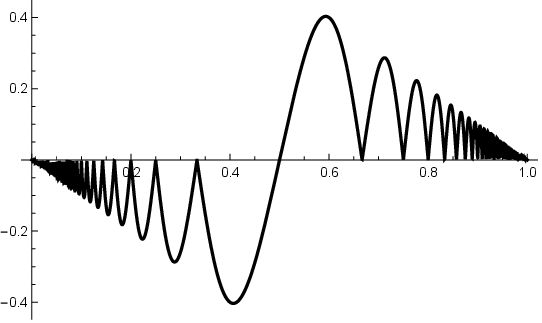,height=1in,width=2in}
\caption{The function $F$: countably many isolated points in $S(F)$.}\label{fig:diana}
\end{figure}

In this case we have $S(F)=[0,\frac{1}{2}]\cup \{1\}\cup  \bigcup_{i,j=2}^\infty \{1-\frac{1}{i}-\frac{1}{j}\}$.  

This idea can be used to create a whole class of examples as required in \cite{Diana2024}.
We can take for example a countable set $C=\{c_1,c_2,\cdots, \}$ in the set $J_1$ with the property $J_1\setminus C$ is still open. This is equivalent to requiring that the accumulation points of $C$ can be only $1$ or $1/2$. Then we consider the 
null Hopf set $H=(0,1/2)\cup (1/2,1)\setminus C$ and we construct the Hopf function for this set, say $G$.  We then define 
\[
F_c(x) =
\begin{cases}
 G(2x), & 0 \le x \le \tfrac{1}{2}, \\[6pt]
 -G(2-2x), & \tfrac{1}{2} <  x 
\le  1.
\end{cases}
\]
One can check that this function $F_c$ satisfies the requirements.

\subsection{Case $\bf m=2$.} If $m=2$, let us take $W_1=(\frac{1}{3},\frac{1}{3}+s) \cup (\frac{1}{2}-t,\frac{1}{2})$ with $t+s<\frac{1}{6}$.
Then, we have 
 $$2W_1=W_1+W_1=(\frac{2}{3},\frac{2}{3}+2s)\cup (\frac{5}{6}-t,\frac{5}{6}+s)\cup  (1-2t,1)=(\frac{2}{3},1)$$
 if $2s+t>\frac{1}{6}$ and $s+2t>\frac{1}{6}$. This happens for a lot of values of $t$ and $s$, for instance if $s=t\in (\frac{1}{18},\frac{1}{12})$. Such a set, $W_1\cup 2W_2\cup (1,\infty)$ is not conull, but we can add the set $ Z=(\frac{1}{2},\frac{2}{3})\setminus \overline{\phi(W_1)}=(\frac{1}{2}+t,\frac{2}{3}-s)$ to make the resulting set a full measure $\frac{1}{2}$ and preserve the additivity since any two numbers $x\in Z$ and $y\in W_1\cup 2W_1$ give $x+y$ either in $(1,\infty)$ or in $(\frac{2}{3},1)$.
This idea can be generalized as follows. 

\begin{theorem}\label{maxHopf} For any Hopf set $U$ which has the property that $m(V(U))=m(\overline{V(U)})$, there exists a conull Hopf set which contains it.  
\end{theorem}

\par\vspace{0.1in}

\begin{proof} For a set of real numbers $S$, we denote the topological closure of $S$, as usual, by $\overline{S}$. Let $U=V\cup (1,\infty)$ a Hopf set with $V\subset (0,1)$ and define $A=V\cap (0,\frac{1}{2})$ and $B=(\frac{1}{2},1)\setminus \overline{\phi(A)}$. We want to prove that $\overset{\sim}{U}=A\cup B\cup (1,\infty)$ is a conull Hopf set containing $U$.  It is clear that $\overset{\sim}{U}$ is an open set. By hypothesis on  $U$, we get  $m(A)=m(\overline{A})$.  Hence, we have
$$m(B)=m((\frac{1}{2},1))-m( \overline{\phi(A)})=\frac{1}{2}-m(A) \implies m(A\cup B)=\frac{1}{2},$$
\n which means  $\overset{\sim}{U}$ is maximal in terms of measure. To prove  that $\overset{\sim}{U}$ contains $U$, let us assume by way of contradiction that $U\setminus \overset{\sim}{U}\not = \emptyset $. The only way this can happen is if $\overline{\phi(A)}\cap U \not = \emptyset$. So, let $z$ be a point in this intersection. 

First, let us assume that $z\in \phi(A)$. Then $z=1-a$ for some $a\in A$. Then $1=a+z\in U$ by additivity of $U$, a contradiction with $U\subset (0,1)\cup (1,\infty)$. 

If $z$ is in $\overline{\phi(A)}$ and in $U$, because $U$ is open,  there must be some $\epsilon >$ so that 
$I=(z-\epsilon,z+\epsilon)\subset U$. Because $z\in  \overline{\phi(A)}$, then $z_n\to z$ for some sequence $\{z_n\}$ in $\phi(A)$. Hence, for $n$ big enough, $z_n\in I\subset U$, which contradicts the case we just analyzed earlier. 

We conclude that $U\subset \overset{\sim}{U}$. It remains to show that $\overset{\sim}{U}$ is additive. For this purpose, let $x,y\in \overset{\sim}{U}$ be some arbitrary elements.  We have the following possibilities.
\par\vspace{0.1in}

\n {\bf Case (i): $\bf x,y\in (\frac{1}{2},1)$. } In this situation $x+y>1$ and so $x+y\in  \overset{\sim}{U}$. 
\par\vspace{0.1in}
\n {\bf Case (ii): $\bf x,y\in (0,\frac{1}{2})$. } So, $x$ and $y$ are in $A$. Then $x+y$  is in $U\subset \overset{\sim}{U}$.  
\par\vspace{0.1in}

\n {\bf Case (iii): $\bf x\in A$ and $\bf y\in B$. } Let us assume by way of contradiction that $x+y\in  \overline{\phi(A)}$. This means $x+y=1-w$ with $w_n\to w$ and $\{w_n\}$ a sequence in $A$.

Because $x\in A\subset V$ and $V$ is open, then there exists an interval $(x-\eta,x+\eta)\subset V$ for some $\eta>0$. Using the additivity of $U$, we have  $(x+w_n-\eta,x+w_n+\eta)\subset U$ for all $n$. Every open set on the real line is a finite or at most countable union of disjoint open intervals. So, there are finitely many such intervals in $U$ that have a length greater than a given bound since $U$ contains $(1,\infty)$.  The intervals $(x+w_n-\eta,x+w_n+\eta)$ have length $2\eta$. Hence, for $n$ big enough, the intervals $(x+w_n-\eta,x+w_n+\eta)$ are included in the same open interval of  $U$, say $(a,b)$. 

Letting $n$ go to infinity, we obtain that  $[x+w-\eta,x+w+\eta]\subset [a,b]$.  This shows that actually $x+w\in (a,b)\in U$.  Now, $x+w=1-y<\frac{1}{2}$ and so $x+w\in A$. This implies $y=1-(x+w)\in \phi(A)$ in contradiction with $y\in B$. It remains that $x+y\not \in  \overline{\phi(A)}$ which implies $x+y>1$ or $x+y\in B$. In either case, we have $x+y\in \overset{\sim}{U}$ closing the argument for additivity of $\overset{\sim}{U}$. 
\end{proof}

\begin{corollary}\label{picksinwn} For $n\in \mathbb N$, $n\ge 2$, and $W$ a set in $(\frac{1}{n+1},\frac{1}{n})$ which is finite union of open intervals,  we define $A=\bigcup_{j=1}^k jW$,  where $k=\lfloor \frac{n}{2}\rfloor$ 
and $B=(\frac{1}{2},1)\setminus \overline{\phi(A)}$. Then the set $\overset{\sim}{U}=A\cup B\cup (1,\infty)$ is a conull Hopf set. 
\end{corollary}
\par\vspace{0.1in} 
\begin{proof} If $W_1$ and $W_2$ are finite union of open intervals, then it is easy to see that $W_1+W_2$ is also a finite union of open intervals. We let $V=\bigcup_{j=1}^n jW$ and we observe that, 
by our choice of $k=\lfloor \frac{n}{2}\rfloor$, $A=V\cap (0,1/2)$ and $U=V \cup (1,\infty)$ is then open and additive, i.e., an Hopf set. By our construction in the Theorem ~\ref{maxHopf}, the set $\overset{\sim}{U}$ is exactly the one obtained from $U$. \end{proof}

In particular, if $n=2$ and $W=(\frac{2}{5}, \frac{1}{2})$ then this construction gives $ \overset{\sim}{U}=(\frac{2}{5}, \frac{1}{2})\cup (\frac{3}{5},1)\cup (1,\infty)$. We observe that there is actually a connection with number theory here. If we multiply the set by $10$ we get 
$$10\cdot \overset{\sim}{U}=(4,5)\cup (6,10)\cup (10,\infty),$$
which is a union of three open intervals  that form a conull Hopf set.  In addition, the set is  {\it primitive} in the sense that it cannot be divided by some positive integer and still have integer endpoints.   Let's call such a set a {\it integer multiple primitive conull Hopf} set. 

There exists only one integer multiple primitive conull Hopf set with two intervals: $(1,2)\cup (2,\infty)$. Indeed, if $K=(a,b)\cup (c,\infty)$ is such a set, then $c=2(b-a)$ and $2(a,b)\subset (c,\infty)$. Then this implies $c=2b-2a\ge 2a$ or $b\ge 2a$. But we cannot have $b>2a$ because that gives $b\in (2a,2b)$ and so $b$ is the sum of two elements in $K$, a contradiction. It remains that $b=2a$ and by dividing everything by $a$ we obtain the result. 

\par\vspace{0.1in} For three intervals, it is not hard to see that the integer multiple primitive conull Hopf sets are exactly the sets from the Introduction,
$$(a,b)\cup (2b-a,2b)\cup (2b,\infty),$$
\n for all positive integers $a$, $b$ relatively prime, so that $b\in (a,\frac{3a}{2}]$. For $a=4$ we get $b\in (4,6]$, and the only possible value is $b=5$, which gives the set $10\cdot\overset{\sim}{U}$ above. For $a=5$ we obtain two sets:
$$(5,6)\cup (7,12)\cup (12,\infty)\  \ \text{and}\ \  (5,7)\cup (9,14)\cup (14,\infty). $$
An example with four intervals can be obtained as in Corollary~\ref{picksinwn}: for $n=2$ and $W=(\frac{1}{3},\frac{5}{12})$ we get $A=W$ and $B=(\frac{1}{2},\frac{7}{12})\cup (\frac{2}{3},1)$, and multiplying by $12$ we obtain the set from the Introduction,
$$H(h)=(4,5)\cup (6,7)\cup (8,12)\cup (12,\infty).$$
A graph of a corresponding function $h$ is shown in Figure~\ref{fig:fourint}.

\begin{figure}[htbp]
\centering
\epsfig{file=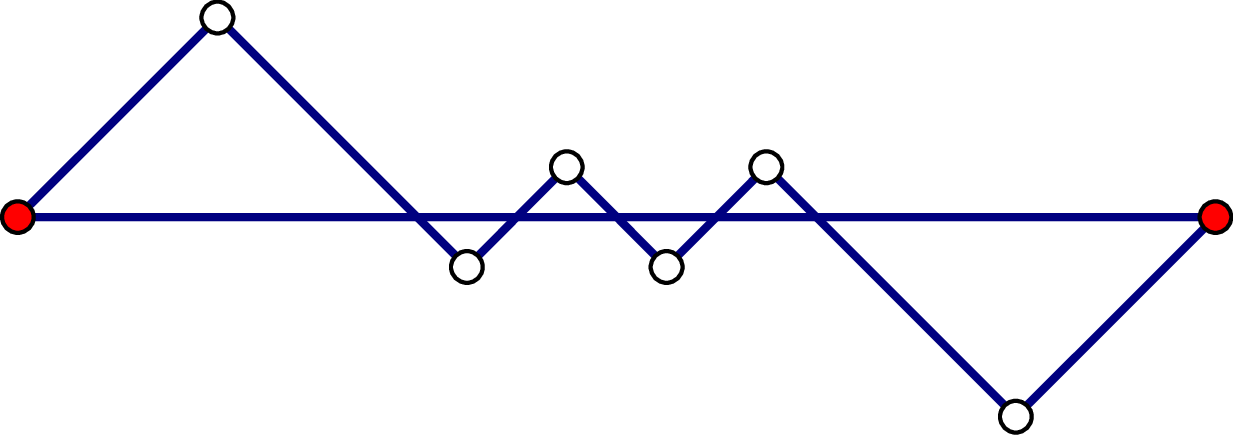,height=1in,width=2in}
\caption{A function $h$ with $H(h)=(4,5)\cup (6,7)\cup (8,12)\cup (12,\infty)$.}\label{fig:fourint}
\end{figure}

\par\vspace{0.1in}  It is natural to wonder about the answer to the  following question.

\n {\bf Question 1:}  {\it  ``What are all integer multiple primitive conull Hopf sets with $n$ intervals?"}

\section{Symmetry of the conull Hopf sets.}\label{sec:symmetry}
We observe that the examples of conull Hopf sets in the Introduction have  a ($180^{\circ}$  rotational) symmetry with respect to 1/2 of the standard Hopf functions associated to them. Let's make this precise.

\begin{theorem}\label{symmetry} Let us consider a function $f\in {\cal F}$ with  $U=H(f)$,  with $U=V\cup (1,\infty)$,  a conull Hopf set.  Then we have $S(f)=\overline{\phi(V)}\cup  \partial V \cup \{0,1\}$ (where $\partial V$ is the boundary of $V$).
\end{theorem}

\begin{proof} We prove this by double inclusion. Since $S(f)=[0,1]\setminus V$ and $\phi(V)\cap V=\emptyset$, it follows that  $\phi(V)\subset S(f)$.  Applying the closure endofunctor gives $\overline{\phi(V)}\subset \overline{S(f)}=S(f)$  since $S(f)$ is a closed set. On the other hand $\partial S(f)=\partial V$ which implies $\partial V \subset S(f)$. We have then one inclusion:
$$\overline{\phi(V)}\cup  \partial V \cup \{0,1\}\subset S(f),$$
\n since we have proved that $\{0,1\}\subset S(f)$. 

 For the other inclusion, we start again with $\phi(V)\cap V=\emptyset$ and first conclude that  $$m(\phi(V)\cup V)=m(\phi(V))+m(V)=\frac{1}{2}+\frac{1}{2}=1.$$ Hence, we must also have $m(\overline{V} \cup \overline{\phi(V)})=1$ as $\overline{V} \cup \overline{\phi(V)}\subset [0,1]$. 
Let us consider the set $$X=(0,1)\setminus( \overline{V} \cup \overline{\phi(V)}),$$ 
which is by construction an open set. But with what we have shown earlier $m(X)=0$. Therefore, $X$ being open and of measure zero, it must be the empty set. As a result, we have 
$$[0,1]=  \overline{V} \cup \overline{\phi(V)}) \cup \{0,1\}=V\cup \partial V  \cup \overline{\phi(V)} \cup \{0,1\},$$

\n which proves the other inclusion. 
\end{proof}
In particular, if $H(f)$ is a finite union of intervals, the sets $\partial V$ and $\partial \phi(V)$ are finite sets. So, we get that essentially $S(f)$ is $\phi(V)$, and so the symmetry we observed follows from the construction of $f$ (done in \cite{Hopf}, and described here in the proof of Theorem~\ref{hopf}). 
\section{Conclusions and open questions}\label{incheiere}
The concepts we introduced and studied here appear to be new. The mathematics literature contains well known notions such as Hopf algebras and Hopf fibrations but it is quite surprising that the Hopf's theorem in 1937 (Theorem~\ref{hopf}) didn't open up the algebraic lines of investigations that we just barely looked into in  this article. The connections with various branches of mathematics are clearly visible and we pointed out to a few: number theory, additive number theory, algebraic combinatorics, algebra, analysis and topology, descriptive set theory and one may think of many others. Theorem~\ref{maxHopf} begs for a more general statement and we believe that under a some mild restrictions it can be generalized to more general unions of open sets. In connection with integer multiple primitive conull Hopf sets, one can ask a lot more questions. Another one is  also quite natural: {\it ``are there examples of such sets which are not the result of our construction in Corollary~\ref{picksinwn} ?"} A closer look at the example in (\ref{newexample}) shows that the answer is ``Yes" and so one can investigate other ways of constructing integer  Hopf sets. 

We close with a problem that generated our interest initially for this topic.  Let us consider the function $g_n(x)=\sum_{k=1}^n \sin (2\pi kx)$, $x\in [0,1]$.
\par\vspace{0.1in}

\n {\bf Question  2.} Is it true that  $H(g_n)$ is the set $V_n$ defined in (\ref{generichopf})?

In Figure~\ref{fig:fourint}, we observe that for certain chord lengths there exist several ways to obtain them on the graph of $h$.  
 In connection to this, there is a problem of  D. J. Newman proposed in 1963 in this Monthly with a solution by N. J. Fine (\cite{Newman1963}) in the following year, asking equivalently the following.

 \begin{theorem}\label{Newman63}
For each $n \in \mathbb{N}$ and each $f \in \mathcal{F}$, 
there exist at least $n$ distinct chords of $f$ whose lengths are integer multiples of $\frac{1}{n}$.
\end{theorem}
 
We invite the readers to prove this by induction using the UCT. We observe that for each $\ell \in {\cal P}$ we can define a vector $v(\ell)=[m_1,m_2,...,m_n]$ with $m_i$ non-negative integers that represent the number of instances the chord of length $\ell$ appears on the graph of $f$.  Clearly, we have  $m_1\ge 1$ by UCT, $m_n=1$ and $\sum_{k}m_k\ge n$ by Theorem~\ref{Newman63}. Knowing this information tells us more about the function $f$, so a natural question here is if this vector function on $\cal P$, determines the function $f$ up to certain transformations that leave the set $S(f)$ invariant (such as those described at the beginning of Section~\ref{nullHopfsets}).

 There is definitely an analogy to the Frobenius problem, about postage stamps: given a number of postage stamps having integer values  (even more than 3, the problem is still open), what is the smallest number $n$ that any value bigger than $n$ can be covered by a combination of the postage stamps? For a Hopf set one can start with a number of open intervals,  and take the smallest open set $U$ containing these intervals making sure it is closed under addition. Then one question is to determine $\alpha_U$.

\section{ACKNOWLEDGMENT.}   We thank the referees, whose suggestions improved our article in many ways: some of the theorems got stronger, and the exposition has more transition points. We really liked the comment about the proof of Hopf's Theorem: ``short, sweet, and instructive". I hope our readers will find that too.

\bibliographystyle{vancouver}
\bibliography{VancouverExamples}

@Article{	  Ampere,
author   = {Ampère, André-Marie},
  title		= { Recherches sur quelques points de la theorie des fonctions derivees qui conduisent a une nouvelle demonstration de la serie de
{T}aylor, et a l’expression finie des termes qu on neglige lorsqu on arrete
cette serie a un terme quelconque},
  journal	= { J. Ecole Polytechique},
  volume	= {6},
  number	= {13},
  pages		= {148--181},
  year		= 1806
}

@Article{	  Levy,
author   = {L\'evy, Paul},
  title		= { Sur une G\'en\'eralisation du Th\'eor\`eme de {R}olle},
  journal	= { C. R. Acad. Sci.},
  volume	= {198},
  number	= {13},
  pages		= {424--425},
  year		= 1934
}

@Article{	  Hopf,
author   = {Hopf, Heinz},
  title		= {{\"U}ber die Sehnen ebener Kontinuen und die Schleifen geschlossener Wege},
  journal	= { Comment. Math. Helv.},
  volume	= {9},
  pages		= {303--319},
  year		= 1937
}

@Article{	Burns,
author   = {Burns, K and Davidovich, O and Davis, D},
  title		= { Average Pace and Horizontal Chords},
  journal	= { Math Intel},
  volume	= {39},
  number	= {4},
  pages		= {41--45},
  year		= 2017
}

@Article{Newman1963,
author   = {Newman, D. J. and Fine, N. J. },
  title		= {An extended chord theorem},
  journal	= {Amer Math Monthly},
  volume	= {71},
  number	= {1},
  pages		= {104},
  year		= 1964
}

@Article{Diana2024,
author   = {Davis, D.  and Troubetzkoy, S. },
  title		= {The horizontal chord set},
  journal	= {Amer Math Monthly},
  volume	= {131},
  number	= {6},
  pages		= {519-525},
  year		= 2024
}

@Article{Levit,
author   = {Levit, R. J.  },
  title		= {The finite difference extension of {R}olle's {T}heorem},
  journal	= {Amer Math Monthly},
  volume	= {70},
  number	= {1},
  pages		= {26-30},
  year		= 1963
}

@Book{Edgar,
  author	= {Edgar, G.},
  title		= {Measure, Topology, and Fractal Geometry},
  publisher	= {Springer},
  address	= {},
  volume  = {},
  year = {1990},
  series  = {}
  }

@Book{Falconer,
  author	= {Falconer, K.},
  title		= {Fractal Geometry: Mathematical Foundations and Applications},
  publisher	= {Wiley},
  address	= {},
  volume  = {},
  year = {2014},
  series  = {3rd ed.}
  }

\bigskip
{\raggedright
\noindent Ion Ciudin, Colegiul Na\c{t}ional ``Mihai Eminescu'', Boto\c{s}ani, Romania, Email Address: ciudini@yahoo.com

\vspace{15px}

\noindent Eugen J. Ionașcu, Columbus State University, Columbus, Georgia, Email Address: nnueminescu@outlook.com
\par}
\end{document}